\documentclass{amsart}

\usepackage{amsmath}
\usepackage{dutchcal}
\usepackage{paralist}
\usepackage[all]{xy}

\usepackage[misc]{ifsym}
\usepackage{epsfig} 
\usepackage{epstopdf} 
\usepackage[colorlinks=true]{hyperref}
\hypersetup{urlcolor=blue, citecolor=green}
\allowdisplaybreaks

\newtheorem{theorem}{Theorem}[section]
\newtheorem{lemma}[theorem]{Lemma}

\theoremstyle{definition}
\newtheorem{definition}[theorem]{Definition}
\newtheorem{example}[theorem]{Example}

\newtheorem{corollary}[theorem]{Corollary}

\newtheorem{proposition}[theorem]{Proposition}

\theoremstyle{definition}

\newtheorem{remark}[theorem]{Remark}

\theoremstyle{remark}

\numberwithin{equation}{section}

\begin{document}

\title{Shadowing property for set-valued map and its inverse limit}


\author{Zhengyu Yin}
\address{School of Mathematics, Nanjing University, Nanjing 210093, People's Republic of China}
\curraddr{}
\email{yzy199707@gmail.com}
\thanks{}


\date{}

\dedicatory{}

\commby{}

\begin{abstract}
In this article, we investigate the relationship between the shadowing property of set-valued maps and their associated inverse limit systems. We show that if a set-valued map is expansive and open in the context of set-valued dynamics, then certain induced inverse limit systems have the shadowing property. Additionally, we prove that a continuous set-valued map has the shadowing property if and only if some of its induced inverse limit system also has shadowing property. Finally, we establish that the shadowing property of a set-valued map is equivalent to the shadowing property of its induced inverse set-valued system.
\end{abstract}

\maketitle

\section{Introduction}
The shadowing property, also known as pseudo-orbit tracing property, is a key concept in dynamical systems that describes the stability of orbits under small perturbations \cite{Wa}. This property plays a crucial role in both stability and ergodic theory \cite{Bl, Bo, Wa, Wa1}, and has been extensively studied in a variety of contexts (see, e.g., \cite{CP, GM0, Pi, PR}). Recently, in \cite{GM}, Good and Meddaugh uncovered a fundamental connection between the shadowing property and inverse limits of inverse systems satisfying the Mittag-Leffler condition, particularly those consisting of shifts of finite type.

For single-valued maps, Sakai \cite{Sa} explored various shadowing properties for positively expansive maps, proving that an expansive map possesses the shadowing property if and only if it is open. In the case of homeomorphisms, Lee and Sakai \cite{LS} demonstrated that different shadowing properties are equivalent for expansive systems. Considering the inverse limit, Chen and Li \cite{CL} showed that a continuous map has the shadowing property if and only if its inverse limit also possesses shadowing property.

For set-valued maps, the inverse limit concept is similarly important, particularly in the context of topological and dynamical properties such as continuum theory \cite{Ma, IM, KN} and chaos \cite{IM1}. Some related topics also appear in these two excellent and rich books \cite{AF,HP} written by Aubin and Frankowska, Shouchuan Hu and Papageorgiou, respectively. For dynamical properties, recently, Kelley and Tennant \cite{KT} introduced the notion of topological entropy for set-valued maps, analogous to the classical topological entropy for single-valued maps. Raines and Tennant \cite{RT} extended the specification property from single-valued maps to set-valued maps, proving that if a set-valued map has the specification property, it is topologically mixing and has positive entropy.
Cordeiro and Pacifico \cite{CP} defined the concept of continuum-wise expansiveness for set-valued maps and proved if a set-valued map on positive dimension space has $cw$-expansive then the set-valued map has positive entropy.
In \cite{MRT},  Metzger,  Morales, and Thieullen. established that a set-valued map satisfying positive expansiveness (as defined by Williams \cite{Wi}) and the pseudo-orbit tracing property (shadowing property, as defined by Pilyugin and Rieger \cite{PR}) ensures that its inverse limit is topologically stable in the sense of set-valued maps.

In this paper, we focus on the shadowing property in set-valued systems and its relationship with inverse limits. Our goal is to investigate the conditions under which a set-valued map with the shadowing property implies that its induced inverse limit also possesses the shadowing property in the sense of single-valued systems.

The paper is organized as follows: In Section 2, we review key definitions and fundamental properties. Section 3 investigates the relationship between openness of set-valued maps and their induced inverse limit actions. In Section 4, we explore the shadowing property for set-valued maps and their corresponding inverse limit systems.

\section{Set-valued maps}

Let $(X,d)$ be a metric space. For each $x \in X$ and $A \subset X$, we denote $\overline{A}$ as the closure of $A$, and define the distance from a point $x$ to a set $A$ as 
$
d(x,A) = \inf \{d(x,a) : a \in A\}.
$
For any $\varepsilon > 0$, we define the open ball centered at $x$ with radius $\varepsilon$ as 
$
B_d(x, \varepsilon) = \{y \in X : d(x,y) < \varepsilon\},
$
and for a set $A \subset X$, we define the $\varepsilon$-neighborhood of $A$ as 
$
B_d(A, \varepsilon) = \{y \in X : d(y, A) < \varepsilon\}.
$

We denote by $2^X$ the collections of all nonempty compact subsets of $X$. Given a compatible metric $d$ on $X$, the \textit{Hausdorff metric} $d_H$ on $2^X$ is defined by 
\[
d_H(A,B) = \max\{\sup_{a \in A} d(a,B), \sup_{b \in B} d(A,b)\}.
\]
It is known that the topology induced by the Hausdorff metric on $2^X$ is independent of the choice of the compatible metric $d$ on the space $X$.

\begin{definition}\label{def1.1}
Let $X$ be a metric space with metric $d$. A set-valued map $F : X \to 2^X$ is defined as follows: 
\begin{enumerate}[(a)]
    \item $F$ is \textit{upper semicontinuous} if for every $x \in X$ and every open set $U \supset F(x)$, there exists an open neighborhood $V$ of $x$ such that $F(y) \subset U$ for all $y \in V$;
    \item $F$ is \textit{lower semicontinuous} if for any $x \in X$ and an open set $U$ intersecting $F(x)$ (i.e., $F(x) \cap U \neq \emptyset$), there exists an open neighborhood $V$ of $x$ such that $F(y) \cap U \neq \emptyset$ for all $y \in V$;
    \item Moreover, $F$ is \textit{continuous} if it is both upper and lower semicontinuous.
\end{enumerate}
\end{definition}

If $F : X \to 2^X$ is a set-valued map on $X$, we call the pair $(X, F)$ a \textit{set-valued system}.
In this paper, we use the notation $f$ to denote a single-valued map and $F$ to denote a set-valued map. Unless otherwise stated, we assume that $X$ is a compact metric space.

For a subset $A \subset X$, we denote
\[
F(A) := \{x \in X : x \in F(a) \text{ for some } a \in A\} = \bigcup_{a \in A} F(a),
\]
and the inverse image of $F$ by 
\[
F^{-1}(A) := \{x \in X : F(x) \cap A \neq \emptyset\}.
\]
Next, we recall some basic properties for set-valued maps, some of which may have appeared in previous works (see e.g., \cite{AF, HP,Mi}).

\begin{proposition} \rm{(\cite{AF,HP})}\label{pro3}
    A set-valued map $F$ on $X$ is upper semicontinuous if and only if for any sequence $(x_n, y_n)$ with $y_n \in F(x_n)$ that converges to a point $(x,y) \in X \times X$, it follows that $y \in F(x)$.
\end{proposition}

\begin{proposition}\label{prop1.2}
Let $(X, F)$ be a set-valued system and $d$ a compatible metric on $X$. Then $F$ satisfies the following:
\begin{enumerate}[(a)]
    \item $F$ is upper semicontinuous if and only if $F^{-1}(A)$ is closed for each $A \in 2^X$;
    \item $F$ is lower semicontinuous if and only if $F^{-1}(B)$ is open for each open subset $B$ of $X$;
    \item $F$ is continuous if and only if $F(x_i)$ converges to $F(x)$ under Hausdorff metric whenever $x_i$ converges to $x$ in $X$.
\end{enumerate}
\end{proposition}

Let $f$ be a single-valued map on $X$. Recall that $f$ is \textit{open} (resp. \textit{closed}) if $f$ maps open (resp. closed) subset to open (resp. closed) subset As an application of the continuity of a set-valued map, we have the following known lemma. For completeness, we include proof here.

\begin{lemma} (Continuity of fiber map)\label{lem2.66}
    Let $f: X \to Y$ be a continuous map between compact metric spaces. Then $f$ is open if and only if the fiber set-valued map $F: X \to 2^X$ defined by 
    $F(x) = f^{-1}(f(x))$ 
    is continuous.
\end{lemma}
\begin{proof}
    Suppose that $f$ is open. Let $U \subset X$ be an open subset. Then 
    \[
    x \in F^{-1}(U) \Longleftrightarrow f^{-1}(f(x)) \cap U \neq \emptyset \Longleftrightarrow f(x) \in f(U) \Longleftrightarrow x \in f^{-1}(f(U)).
    \]
    Thus, $F^{-1}(U) = f^{-1}(f(U))$ is open, proving that $F$ is lower semicontinuous.

    Let $A \in 2^X$. By a similar argument, $F^{-1}(A) = f^{-1}(f(A))$ is closed. Therefore, $F$ is upper semicontinuous. Hence, by Proposition \ref{prop1.2} $F$ is continuous.

    Conversely, suppose $F$ is continuous. An equivalent description of a continuous map being open is that for any closed subset $A \subset X$, the union of all fibers contained in $A$ is itself closed (see \cite[Theorem 3.10]{Ke}). Notice that the set $Q$ consisting of points such that $f^{-1}(f(x)) \subset A$ is equal to $( F^{-1}(A^c) )^c$ and as $F$ is lower semicontinuous, $Q$
    is closed. Thus, $f$ is open.
\end{proof}

\section{Inverse limit  and openness}
Let $(X,d)$ be a compact metric space and $\mathbb{A}=-\mathbb{N}\text{ or }\mathbb{N}$. Let $X^\mathbb{A}$ be the space endowed with the product topology. We consider the compatible metric $\rho$ on $X^\mathbb{A}$ which is defined by
\[\rho((x_n),(y_n))=\sum_{n\in\mathbb{A}}^\infty\displaystyle\frac{d(x_n,y_n)}{2^{|n|+1}},\text{ for all }(x_n),(y_n)\in X^\mathbb{A},\]
and let $\rho_H$ be the corresponding Hausdorff metric on $2^{X^\mathbb{A}}$ induced by $\rho.$

As $X$ is compact, for convenience, we all assume that the diameter of $X$ is $1$ under its compatible metric $d$.

In \cite{KT, RT}, the authors introduced various orbit spaces in the study of their connections with topological entropy. We recall the definitions of these orbit spaces:
\begin{align*}
    &{\rm{Orb}}_l(X)=\{(\cdots,x_{-1},x_0)\in X^{-\mathbb{N}}:x_i\in F(x_{i+1})\},\\
     &{\rm{Orb}}_r(X)=\{(x_0,x_1,\cdots)\in X^\mathbb{N}:x_{i+1}\in F(x_i)\},\\
     &{\rm{Orb}}_{inv}(X)=\{(x_0,x_1,\cdots)\in X^\mathbb{N}:x_i\in F(x_{i+1})\}.
\end{align*}
 Due to the nice property of upper semicontinuity (see Proposition \ref{pro3}), if $F$ is an upper semicontinuous set-valued map, then these orbit spaces are all closed subsets of  $X^\mathbb{A}$. 
 Furthermore, if $F$ is a single-valued continuous map, then ${\rm{Orb}}_{inv}(X)$ coincides with the inverse limit in the classical sense (see \cite{CL}).

For convenience, we denote points in ${\rm{Orb}}_l(X)$, ${\rm{Orb}}_r(X)$ and ${\rm{Orb}}_{inv}(X)$ by $(x_{-n}), (x_{n})$ and $[x_n]$, respectively. Denote $\pi_k$ by the $k$th projection from ${\rm{Orb}}_r(X)$ to $X$, and by $\varphi_k$ the $k$th projection from ${\rm{Orb}}_{inv}(X)$ to $X$, that is, $\pi_k((x_n))=x_k$ and $\varphi_k([y_n])=y_k$.

With the above notations above, we can define the following single-valued and set-valued maps, respectively.
\begin{align*}
    &{\sigma}_{lF}:{\rm{Orb}}_l(X)\to {\rm{Orb}}_l(X)\text{ by }(x_{-n})\mapsto (x_{-n-1}),\\
    &{\sigma}_{rF}:{\rm{Orb}}_r(X)\to {\rm{Orb}}_r(X)\text{ by }(x_{n})\mapsto (x_{n+1}),\\
    &{F}_{inv}:{\rm{Orb}}_{inv}(X)\to 2^{{\rm{Orb}}_{inv}(X)}\text{ by }[x_n]\mapsto F(x_0)\times [x_n].
\end{align*}
It is easy to check that $F$ is upper semicontinuous (continuous) if and only if $F_{inv}$ is upper semicontinuous (continuous).

Similar to single-valued case, we say that a set-valued map $F$ is \textit{closed} if $F(A)\in 2^X$ for all $A\in 2^X$ and \textit{open} if $F(U)$ is open for any open subset $U$ in $X$.

\begin{lemma}\rm{(\cite{AF,HP})}\label{lem2.3}
  If $F$ is an upper semicontinuous set-valued map then $F$ is closed.
\end{lemma}

Then, we prove a technique lemma for further use.

\begin{lemma}\label{lem3.12}
    Let $(X,F)$ be a continuous set-valued system with metric $d$ on $X$. For each $\delta>0$, there is $\delta_1>0$ such that if $d(x,y)<\delta_1$, then $\rho_H(\pi_0^{-1}(x),\pi_0^{-1}(y))<\delta$, where $\rho_H$ is the Hausdorff metric on $2^{Orb_r(X)}$. 
\end{lemma}
\begin{proof}
    Choose $N\in \mathbb{N}$ such that 
    $\sum_{n=N}^\infty{1}/{2^n}<\delta/2.$
    Since $F$ is continuous, by Proposition \ref{prop1.2} there exists a sequence $\{\delta_1,\cdots,\delta_{N-1}\}$, where $\delta_i<\delta/4$ satisfying: 
     \begin{enumerate}[(N-1)]
     \item if $d(x,y)<\delta_1$, then $d_H(F(x),F(y))<\delta_2,$
     \item if $d(x_{i-1},y_{i-1})<\delta_i$ then $d_H(F(x_{i-1}),F(y_{i-1}))<\delta_{i+1}$ for all $2\leq i\leq N-2$,
     \item if $d(x_{N-2},y_{N-2})<\delta_{N-1}$, then $d_H(F(x_{N-2}),F(y_{N-2}))<\delta/4$.
     \end{enumerate}
     Let $(y_n)\in \pi_0^{-1}(y)$. Since $d(x,y)<\delta_1$ we can find $x_1\in F(x)$ such that 
    $ d(x_1,y_1)<\delta_2.$
     By induction, we construct a sequence $\{x_1,\cdots,x_{N-1}\}$ such that 
     \begin{enumerate}[(1)]
         \item $x_{i+1}\in F(x_i)$ and $d(x_i,y_i)<\delta_{i+1}$ for $i=1,\cdots,N-2$ and 
         \item $d(x_{N-1},y_{N-1})<\delta/4$.
     \end{enumerate}
    Extend the sequence by taking $x_{j}\in F(x_{j-1})$ for all $j\geq N$. This produces an $F$-orbit $(x_n)$ such that 
     \[\rho((x_n),(y_n))<\sum_{i=0}^{N-1}\displaystyle\frac{d(x_i,y_i)}{2^{i+1}}+\sum_{n=N}^\infty\displaystyle\frac{1}{2^n}<\delta.\]

Similarly, for each $(x_n)\in \pi_0^{-1}(x)$, there is an element $(y_n)\in \pi_0^{-1}(y)$ such that $\rho((x_n),(y_n))<\delta$. Hence, we conclude that $\rho_H\left(\pi_0^{-1}(x),\pi_0^{-1}(y)\right)<\delta.$
\end{proof}

\begin{corollary}\label{cor3.3}
    Let $(X,F)$ be a set-valued system. If $F$ is continuous,
        then the projection map $\pi_0$ is open.
\end{corollary}
\begin{proof}
    For any $(x_n)\in {\rm{Orb}}_r(X)$, we define
    $\widetilde{F}((x_n)):=\pi^{-1}_0(\pi_0((x_n)))=\pi^{-1}_0(x_0).$
    By Lemma \ref{lem3.12} if $F$ is continuous it is easy to see that 
    $\widetilde{F}$ is a continuous set-valued map. Thus, by Lemma \ref{lem2.66}, $\pi_0$ is open.    
\end{proof}

Let $(X,F)$ be a set-valued system. We say that $F$ is \textit{onto} if for each $y\in X$ there is $x\in X$ such that $y\in F(x)$.
Recall that if $F$ is an upper semicontinuous onto set-valued map on $X$, $F^{-1}$ is also a well-defined upper semicontinuous set-valued map on $X$. Thus, We have the following statement.

\begin{theorem}\label{thm3.4}
   Let $(X,F)$ be an upper semicontinuous set-valued system, the following statements hold.
   \begin{enumerate}[(a)]
   \item If $F$ is onto, then $F$ is open if and only if $F^{-1}$ is continuous.
       \item  If $F$ is open then $\sigma_{rF}$ and $\sigma_{lF}$ are also open
     
       \item  $F$ is continuous on $X$. Then ${\sigma}_{rF}$ is open implies that $F$ is open.
   \end{enumerate}
\end{theorem}
\begin{proof}
(a) Fix $x\in X$, if $y\in (F^{-1})^{-1}(x)$ then $x\in F^{-1}(y)$, which means that $y\in F(x)$ and $(F^{-1})^{-1}(x)\subset F(x)$. Conversely, suppose $y\in F(x)$ then $x\in F^{-1}(y)$, which means that $y\in (F^{-1})^{-1}(x)$ and $F(x)\subset (F^{-1})^{-1}(x)$.
Thus, by Proposition \ref{prop1.2}, $F$ being open implies that $F^{-1}$ is lower semicontinuous. Also, if $F^{-1}$ is lower semicontinuous, by Proposition \ref{prop1.2}, $F$ is an open map.

(b) Let $(x_n)\in {\rm{Orb}}_r(X)$ and $\varepsilon_0>0$. Since $F$ is open, there is $0<\varepsilon_1<\varepsilon_0$ such that 
 $ B_d(x_1,\varepsilon_1)\subset F(B_d(x_0,\varepsilon_0)).$
  By induction, for each $n\in\mathbb{N}$ there is $0<\varepsilon_{n}<\varepsilon_{n-1}$
  such that 
  $B_d(x_n,\varepsilon_n)\subset F(B_d(x_{n-1},{\varepsilon_{n-1}})).$
Taking open neighborhood $U_n$ of $(x_n)$ with the form
 \[ U_n:=(B_d(x_0,\varepsilon_0)\times\cdots\times B_d(x_n,\varepsilon_n)\times \Pi_{i=n+1}^\infty X)\cap {\rm{Orb}}_r(X)\text{ for }n\in\mathbb{N}.\]
Then we have
 \[(x_{n+1})\in (B_d(x_1,\varepsilon_1)\times\cdots\times B_d(x_{n+1},\varepsilon_{n+1})\times \Pi_{i=n+2}^\infty X)\cap {\rm{Orb}}_r(X)\subset \sigma_{rF}(U_n),\]
 which implies that ${\sigma}_{rF}$ is open.
 The proof for $({\rm{Orb}}_r(X),{\sigma}_{rF})$ is nearly same, and we omit it.

 (c)  This follows directly from Lemma \ref{lem3.12} and Corollary \ref{cor3.3}.
    If $\pi_0$ is open, then for any open subset $U$,
 \[F(U)=\pi_0(\sigma_{rF}((U\times\Pi_{i=1}^\infty X)\cap  {\rm{Orb}}_r(X)))\] is open in $X.$
\end{proof}

\section{Inverse limit and Shadowing property}
First, we review some fundamental shadowing property concepts related to single-valued maps and set-valued maps.

Let $(X,F)$ be a set-valued system and $d$ a metric on $X$. Given $\delta>0$, a sequence $\{x_n\}$ in $X$ is called a \textit{$\delta$-pseudo-orbit} if 
$d(x_{n+1},F(x_n))<\delta\text{ for all }n\in\mathbb{N},$
and $\{x_n\}$ is called an \textit{$F$-orbit} if 
$(x_n)\in {\rm{Orb}}_r(X).$ In particular, if $F$ is a single-valued mapping, these notions apply equally well. 
\begin{definition}\rm{(\cite{Wa})}
    Let $(X,f)$ be a single-valued system. We say that $f$ has \textit{shadowing property} (or \textit{POTP}), if for each $\varepsilon>0$ there is $\delta>0$ such that for every
 $\delta$-pseudo-orbit $\{x_n\}$, there is a point $x\in X$ such that $ d(f^n(x),x_n)<\varepsilon\text{ for all }n\in\mathbb{N}.$ In this case, we say that $\{x_n\}$ is $\varepsilon$-shadowed by $x$.
\end{definition}

\begin{definition}\rm{(cf. \cite{PR})}\label{def3.2}
    Let $(X,F)$ be a set-valued system. The map $F$ is said to have \textit{shadowing property}, if for every $\varepsilon>0$ there is $\delta>0$ such that for each $\delta$-pseudo-orbit $\{x_n\}$ of $F$, there exists an $F$-orbit $(y_n)$ satisfying 
    $d(x_n,y_n)<\varepsilon\text{ for all }n\in\mathbb{N}.$ In this case, we say that $\{x_n\}$ is $\varepsilon$-shadowed by $(y_n)$.
\end{definition}

\begin{definition}\rm{(\cite{MRT})}
  A set-valued system  $(X,F)$ is said to have \textit{finite shadowing} if for every $\varepsilon>0$ there is $\delta>0$ such that for each finite sequence $\{x_0,\cdots,x_n\}$ satisfying $d(x_{i},F(x_{i-1}))<\delta$, there is a finite sequence $\{y_0,\cdots,y_n\}$ with $y_i\in F(y_{i-1})$ such that
    $d(x_i,y_i)<\varepsilon\text{ for all }i=0,\cdots,n.$
\end{definition}
\begin{lemma}\rm{(\cite{MRT})}\label{lem3.8}
    Let $F:X\to 2^X$ be an upper semicontinuous set-valued map. Then $F$ has shadowing property if and only if it has finite shadowing.
\end{lemma}

The following statement shows that multivalued maps with shadowing property can be generated by carrying out symmetry operations on single-valued maps with shadowing property.

\begin{lemma}\label{lem3.17}
    Let $a \geq 0$ and $f$ a continuous single-valued map on $[a,b]$, and $d$ is the Euclidean metric on $[a,b]$. If $f$ has shadowing property, then the set-valued map $F$ defined by 
    \[
    F(x):=
    \begin{cases} 
        \{f(-x), -f(-x)\}, & x \in [-b, -a], \\  
        \{f(x), -f(x)\}, & x \in [a, b],
    \end{cases}
    \]
    also has shadowing property.
\end{lemma}

\begin{proof}
    Suppose $f$ has shadowing property. Let $\varepsilon > 0$ and $\delta>0$ witness $\varepsilon$-shadowing. Let $\{x_n\}$ be a $\delta$-pseudo-orbit of $F$.
    If all $x_n$ are non-negative or all $x_n$ are non-positive, then, since $f$ has the shadowing property (and so does $-f(-x)$ for $x\in [-b,-a]$), there exists a point $y \in  [a, b]$ or $y\in [-b,-a]$ such that either:
    \[
    \{y, f(y), \ldots, f^n(y), \ldots\} \quad \text{or} \quad \{y, -f(-y), \ldots, -f^n(-y), \ldots\}
    \]
    $\varepsilon$-shadows $\{x_n\}$.

    If $\{x_n\}$ contains both positive and negative elements, we proceed as follows. Assume $x_0 < 0$ then we set $x_0' = -x_0$. Since $F(x) = F(-x)$, we still have 
    $d(F(x_0'), x_1) < \delta.$
    Let $x_i$ (for $i \geq 1$) be the first negative element in the sequence. Then we set $x_i' = -x_i$, by a straightforward calculation, we have 
   $d(F(x_{i-1}), x_i') = d(F(x_{i-1}), x_i) < \delta,$
    and
   $d(f(x_{i-1}), x_i') < \delta.$
    By replacing negative terms in the sequence with their positive counterparts iteratively (i.e., $x_i' = x_i$ if $x_i \geq a$ and $x_i' = -x_i$ otherwise), we construct a non-negative sequence $\{x_n'\}$. Then this sequence is a $\delta$-pseudo-orbit of $f$. 
As $f$ has shadowing property, then there exists a point $y \in [a, b]$ such that
    $d(f^n(y), x_n') \leq \varepsilon.$
    Finally, we define
  $y'_n=f^n(y)$ if $x_n\geq a$ and $y_n'=-f^n(y)$, otherwise.
    From the construction of $F$, it is straightforward to verify that $(y_n')$ is an $F$-orbit that $\varepsilon$-shadows $\{x_n\}$.
    Therefore, $F$ has shadowing property.
\end{proof}

\begin{example}\label{3.18}
    Let $f_c$ be the tent map on $[0,2]$ with slope $\sqrt{2}\leq c\leq 2$. Let $F_c$ be the set-valued map induced as in Lemma \ref{lem3.17}.
    It is shown in \cite{PKY} that the tent map $f_c$ has shadowing property for almost all parameters $c$, therefore, the set-valued map $F_c$ also has shadowing property for almost all parameters $c$. 
\end{example}

\begin{theorem}\label{thm4.7}
Let $(X,F)$ be a set-valued system. 
 If $F$ is onto and continuous with shadowing property, then $F^{-1}$ also has shadowing property.

\end{theorem}
\begin{proof}
    Let $\varepsilon>0$ and $\delta>0$ witness the $\varepsilon$-shadowing. Since $F$ is continuous, there is $\delta_1>0$ such that $d_H(F(x),F(y))<\delta/2$ whenever $d(x,y)<\delta_1$, where $d_H$ denotes the Hausdorff metric. 
    
    We claim that $F^{-1}$ has finite shadowing property, thus, shadowing property. Let \(\{x_0, \ldots, x_n\}\) be a \(\delta_1\)-pseudo-orbit of \(F^{-1}\). By definition, for each \(x_i\) there exists \(x_i' \in F^{-1}(x_i)\) such that  
$d(x_i', x_{i+1}) < \delta_1.$
Since \(F\) is continuous, we have  
$d_H(F(x_{i+1}), F(x_i')) < \delta / 2.  $
Therefore,  
$d(x_i, F(x_{i+1})) < \delta \text{ for } i = 0, \ldots, n - 1, $ 
which implies that \(\{x_n, \ldots, x_0\}\) is a \(\delta\)-pseudo-orbit of \(F\).  
By the shadowing property of \(F\), there exists a finite sequence \(\{y_n, \ldots, y_0\}\), where \(y_i \in F(y_{i+1})\) for all \(i = 0, \ldots, n - 1\), such that  
$d(x_i, y_i) < \varepsilon \quad \text{for all } i = 0, \ldots, n.$ 
Then reversing the order of \(\{y_n, \ldots, y_0\}\), we obtain a sequence \(\{y_0, \ldots, y_n\}\) satisfying \(y_{i+1} \in F^{-1}(y_i)\) for all \(i = 0, \ldots, n - 1\) and  
$d(x_i, y_i) < \varepsilon \text{ for all }i = 0, \ldots, n.$ 
Thus, \(\{y_0, \ldots, y_n\}\) \(\varepsilon\)-shadows \(\{x_0, \ldots, x_n\}\), which proves that \(F^{-1}\) has the shadowing property.

\end{proof}

\begin{corollary}
    If $(X, F)$ is a set-valued system such that $F$ is open, onto and continuous, then $F$ has shadowing property if and only if $F^{-1}$ has shadowing property.
\end{corollary}
\begin{proof}
     By Proposition \ref{prop1.2} $F^{-1}$ is a continuous set-valued map, thus, by Theorem \ref{thm4.7}, the conclusion holds.
\end{proof}

\begin{example}
    Let $F_2$ be the set-valued map as defined in Example \ref{3.18}. Then $F_2^{-1}$ also has shadowing property. 
\end{example}

\begin{definition}\rm{(\cite{Wi})}
  A set-valued system $(X,F)$ is called \textit{expansive} (or \textit{positively expansive}) if there exists $\delta>0$ such that $x_0=y_0$ whenever for any two $F$-orbits  $(x_n),(y_n)\in {\rm{Orb}}_r(X)$ with $d(x_n,y_n)<\delta$ for all $n\in \mathbb{N}$.
  The number $\delta$ is called an \textit{expansive constant} for $F$.
 
\end{definition}
\begin{lemma}\rm{(cf. \cite{MRT})}\label{lem4.11}
    Let $(X,F)$ be an upper semicontinuous set-valued system. Then 
         $F$ is expansive if and only if
        $({\rm{Orb}}_r(X),{\sigma}_{rF})$ is expansive.
\end{lemma}
\begin{proof}
$(b)\Longrightarrow (a)$. Suppose ${\sigma}_{rF}$ is expansive with expansive constant $\delta > 0$. Let $(x_n)$ and $(y_n)$ be two distinct $F$-orbits such that
$d(x_n, y_n) < \delta \text{ for all } n \in \mathbb{N}.$
Then,
\[
\rho({\sigma}_{rF}^k(x_n), {\sigma}_{rF}^k(y_n)) = \sum_{n=0}^\infty \displaystyle\frac{d(x_{n+k}, y_{n+k})}{2^{n+1}} < \delta,
\]
which contradicts the expansiveness of ${\sigma}_{rF}$. Thus, $F$ is expansive.

$(a)\Longrightarrow (b).$ Conversely, suppose $F$ is expansive with expansive constant $\delta > 0$. For any two distinct $F$-orbits $(x_n), (y_n) \in {\rm{Orb}}_r(X)$, there exists $k \in \mathbb{N}$ such that
$\delta<d(x_k, y_k).$
Then
\[
\delta/2 < \frac{d(x_k, y_k)}{2} \leq \rho({\sigma}_{rF}^k((x_n)), {\sigma}_{rF}^k((y_n))),
\]
which implies that ${\sigma}_{rF}$ is expansive with expansive constant $\delta/2$.
\end{proof}

In \cite{Sa}, Sakai provides an equivalent condition for when an expansive single-valued map possesses shadowing property, that is.

\begin{theorem}\rm{(\cite[Theorem 1]{Sa})}\label{thm3.10}
    Let $f: X \to X$ be an expansive continuous single-valued map on a compact metric space $X$. Then $f$ is open if and only if $f$ has shadowing property.
\end{theorem}

Combining Sakai's result, we have the following result.

\begin{theorem}\label{thm3.8}
    Let $(X, F)$ be a set-valued system. If $F$ is an expansive upper semicontinuous map, then $F$ is open implying that ${\sigma}_{rF}$ has shadowing property.
\end{theorem}

\begin{proof}
    By Theorem \ref{thm3.4}, the shift map ${\sigma}_{rF}$ on ${\rm{Orb}}_r(X)$ is open. Therefore, by Theorem \ref{thm3.10}, ${\sigma}_{rF}$ has shadowing property.
\end{proof}

Thus, combining Lemma \ref{lem4.11}, we have the following result.

\begin{theorem}
    If $F$ is an open expansive upper semicontinuous map with shadowing property. Then $({\rm{Orb}}_r(X),{\sigma}_{rF})$ is topologically stable.
\end{theorem}

\begin{example}\rm{(\cite[cf. Example 2.2]{MRT})}\label{3.11}
  Let $[0,2]$ endowed with the Euclidean metric $d$ on $[0,2]$. We define a set-valued map $F$ on $[0,2]$ as follows 
  $$
F(x)=\left\{
             \begin{aligned}
             &\{2-2x\}, 0\leq x<1,  \\  
             &\{0,2\}, x=1,\\  
             &\left\{4-2x\right\}, 1<x\leq 2.
             \end{aligned}  
\right.
$$
Here, $F$ is an expansive, upper semicontinuous set-valued open map. By Theorem \ref{thm3.8}, ${\sigma}_{rF}$ has shadowing property, thus, it is topologically stable in the sense of single-valued map.

\end{example}
\begin{remark}
    In \cite{MRT}, the authors presented an example (see \cite[Example 2.2]{MRT}) of a system that satisfies the shadowing property under Definition \ref{def3.2} but did not address whether ${\sigma}_{rF}$ has shadowing property. By Theorem \ref{thm3.8} we know that ${\sigma}_{rF}$ does have shadowing property.
\end{remark}

In \cite{CL}, the authors established that a single-valued map 
$f$ has shadowing property if and only if its associated inverse limit map ${f}_{inv}$ 
  has shadowing property. Here, we investigate shadowing property in the context of (mainly continuous) set-valued maps. 
  
  As Lemma \ref{lem3.8} states, we rely on the finite shadowing property to show that a system has shadowing property.

  For convenience, sometimes, we use the bold symbol $\boldsymbol{x}:=(x_n)$ to denote an element in $
  {\rm{Orb}}_r(X)$ and $\overline{\boldsymbol{x}}:=[x_n]$ to denote an element in ${\rm{Orb}}_{inv}(X)$.

\begin{theorem}\label{thm2.12}
    Let $(X,F)$ be a set-valued system. Then we have:
    \begin{enumerate}[(a)]
        \item If $F$ is continuous, the following are equivalent:
        \begin{enumerate}[(1)]
            \item $F$ has  shadowing property,
            \item ${\sigma}_{rF}$ has  shadowing property,
            \item ${\sigma}_{lF}$ has  shadowing property.
        \end{enumerate}
        \item If $F$ is upper semicontinuous, onto, and open, then $F$ has shadowing property if and only if ${F}_{inv}$ has shadowing property.
    \end{enumerate}
\end{theorem}
\begin{proof}
(a). $(2)\Longrightarrow (1)$ Suppose ${\sigma}_{rF}$ has shadowing property. 
Let $\varepsilon > 0$ and $\delta$ witness $\varepsilon$-shadowing, and let $\delta_1>0$ witness the condition of Lemma \ref{lem3.12} with respect to $\delta$, that is, if $d(x,y)<\delta_1$ and $(y_n)\in {\rm{Orb}}_r(X)$ with $y_0=y$, then there is a point $(x_n)\in {\rm{Orb}}_r(X)$ with $x_0=x$ such that $\rho((x_n),(y_n))<\delta$. Recall that the metric $\rho$ on ${\rm{Orb}}_r(X)$ is defined by 
\[\rho((x_n),(y_n))=\sum_{n\in\mathbb{A}}^\infty\displaystyle\frac{d(x_n,y_n)}{2^{|n|+1}}.\]
Let $\{x_0, \cdots, x_n\}$ be a finite sequence such that 
\[
d(F(x_i), x_{i+1}) < \delta_1, \text{ for all }i=0,\cdots,n-1.
\]
For each $i$, there exists $x_{i-1}' \in F(x_{i-1})$ such that 
\[
    d(x_{i-1}', x_i) < \delta_1.
\]
Let $\boldsymbol{x}_n\in {\rm{Orb}}_r(X)$ with $\pi_0(\boldsymbol{x}_n)=x_n$. By Lemma \ref{lem3.12}, we can find an $F$-orbit $\boldsymbol{x}'_{n-1}\in {\rm{Orb}}_r(X)$ with $\pi_0(\boldsymbol{x}'_{n-1})=x_{n-1}'$ such that 
\[
\rho(\boldsymbol{x}'_{n-1}, \boldsymbol{x}_n) < \delta,
\]
By adding $x_{n-1}$ to the first coordinate of $\boldsymbol{x}_{n-1}'$ obtaining a new point $\boldsymbol{x}_{n-1}\in {\rm{Orb}}_r(X)$ such that 
\[
\rho({\sigma}_{rF}(\boldsymbol{x}_{n-1})=\boldsymbol{x}'_{n-1}, \boldsymbol{x}_n) < \delta.
\]

By repeating the above process, suppose that we have constructed points $\boldsymbol{x}_i,\cdots,\boldsymbol{x}_n\in {\rm{Orb}}_r(X)$. For $i-1$, as $d(F(x_{i-1}),x_i)<\delta_1$, there is $x_{i-1}'\in F(x_{i-1})$ such that $d(x'_{i-1},x_i)<\delta_1$. Hence, there is point $\boldsymbol{x}'_{n-1}$ with $\pi_0(\boldsymbol{x}'_{n-1})=x_{n-1}'$ satisfying $\rho(\boldsymbol{x}_{n-1}',\boldsymbol{x}_i)<\delta$. Similarly, by adding $x_{i-1}$ to the first coordinate of $\boldsymbol{x}_{i-1}'$ obtaining a new point $\boldsymbol{x}_{i-1}\in {\rm{Orb}}_r(X)$ such that 
\[\rho(\sigma_{rF}(\boldsymbol{x}_{i-1},\boldsymbol{x}_i)<\delta.\]
Thus, we obtain a finite sequence of points $\boldsymbol{x}_1,\cdots,\boldsymbol{x}_n$ with $\pi_0(\boldsymbol{x}_i)=x_i$ for all $i$ satisfying 
\[\rho(\sigma_{rF}(\boldsymbol{x}_{i-1}),\boldsymbol{x}_i)<\delta.\]

Finally, since $\sigma_{rF}$ has finite shadowing, there exists $\boldsymbol{y}=(y_n) \in {\rm{Orb}}_r(X)$ such that 
\[
\rho(\sigma_{rF}^j(\boldsymbol{y}),\boldsymbol{x}_j) < \varepsilon \text{ for all } 0 \leq j \leq n.
\]
Then, we have
\[
d(y_j, x_{j}) < 2 \rho({\sigma}_{rF}^j(\boldsymbol{y}), \boldsymbol{x}_j) < 2\varepsilon.
\]
As $\varepsilon$ can be taken arbitrarily, $F$ has the shadowing property.

$(1)\Longrightarrow (2)$ Conversely, suppose $F$ has shadowing property. Let $\varepsilon>0$. Given $\varepsilon_0>0$ such that $\varepsilon_0$ witnesses the condition of Lemma \ref{lem3.12} with respect to $\varepsilon/2$, that is, if $d(x,y)<\varepsilon_0$ then \begin{align}\label{4.1}
    \rho_H(\pi^{-1}_0(x),\pi^{-1}_0(y))<\frac{\varepsilon}{2}.
\end{align}
Since $F$ has shadowing property, let $0<\delta<\min\{\varepsilon_0/2,\varepsilon/2\}$ witness the $\varepsilon_0/2$-shadowing. Now, let $\{\boldsymbol{x}_i\}_{i=1}^k$ be a $\delta/2$-pseudo-orbit in $\rm{Orb}_r(X)$, then from the definition, $\{x_i\}_{i=1}^k=\{\pi_0(\boldsymbol{x}_i)\}_{i=1}^k$ is a $\delta$-pseudo-orbit of $F$. Then there is a finite orbit $(z_i)_{i=1}^k$ such that 
\[d(x_i,z_i)<\frac{\varepsilon_0}{2}.\]

Consider $x_k$ and $z_k$, from \eqref{4.1}, there is an $F$-orbit $\boldsymbol{z}_k$ such that 
\[\rho(\boldsymbol{x}_k,\boldsymbol{z}_k)<\frac{\varepsilon}{2}.\]
By adding $z_{k-1}$ to the first coordinate of $\boldsymbol{z}_k$, we obtain a new $F$-orbit $\boldsymbol{z}_{k-1}$, then 
\begin{align*}
    \rho(\boldsymbol{x}_{k-1},\boldsymbol{z}_{k-1})&\leq \frac{d(x_{k-1},z_{k-1})}{2}+\frac{\rho(\sigma_{rF}(\boldsymbol{x}_{k-1}),\sigma_{rF}(\boldsymbol{z}_{k-1}))}{2}\\
    &=\frac{d(x_{k-1},z_{k-1})}{2}+\frac{\rho(\sigma_{rF}(\boldsymbol{x}_{k-1}),\boldsymbol{z}_k)}{2}\\&
    \leq (\frac{d(x_{k-1},z_{k-1})}{2}+\frac{\rho(\sigma_{rF}(\boldsymbol{x}_{k-1}),\boldsymbol{x}_k)+\rho(\boldsymbol{x}_k,\boldsymbol{z}_k)(=\beta_0)}{2})(=\beta_1)\\
    &\leq\frac{\varepsilon_0}{4}+\frac{\delta/2+\beta_0}{2}\leq \frac{\varepsilon_0}{4}+\frac{\varepsilon/2+\beta_0}{2}.
\end{align*}

By induction, suppose $\boldsymbol{z}_{i}$ has been constructed, by adding $z_{i-1}$ to the first coordinate of $\boldsymbol{z}_i$, we obtain a new $F$-orbit $\boldsymbol{z}_{i-1}$, then 
\begin{align*}
      \rho(\boldsymbol{x}_{i-1},\boldsymbol{z}_{i-1})&\leq \frac{d(x_{i-1},z_{i-1})}{2}+\frac{\rho(\sigma_{rF}(\boldsymbol{x}_{i-1}),\sigma_{rF}(\boldsymbol{z}_{i-1}))}{2}\\
    &=\frac{d(x_{i-1},z_{i-1})}{2}+\frac{\rho(\sigma_{rF}(\boldsymbol{x}_{i-1}),\boldsymbol{z}_i)}{2}\\&
    \leq (\frac{d(x_{i-1},z_{i-1})}{2}+\frac{\rho(\sigma_{rF}(\boldsymbol{x}_{i-1}),\boldsymbol{x}_i)+\rho(\boldsymbol{x}_i,\boldsymbol{z}_i)(= \beta_{k-i})}{2})(=\beta_{k-i-1})\\
    &\leq\frac{\varepsilon_0}{4}+\frac{\delta/2+\beta_{k-i}}{2} \leq\frac{\varepsilon_0}{4}+\frac{\varepsilon/2+\beta_{k-i}}{2}. 
\end{align*}
Hence, by calculating the formula
\[\alpha_{k-i-1}=\frac{\varepsilon_0}{4}+\frac{\varepsilon/2+\alpha_{k-i}}{2}=\frac{\alpha_{k-1}}{2}+\frac{\varepsilon_0+\varepsilon}{4},\]
we obtain
\begin{align}\label{eq4.2}
    \rho(\boldsymbol{x}_{i-1},\boldsymbol{z}_{i-1})\leq (\frac{1}{2})^i(-\frac{\varepsilon}{4})+\frac{\varepsilon+\varepsilon_0}{2}<\varepsilon\text{ for all }i\leq k.
\end{align}
Notice that $\sigma_{rF}^i(\boldsymbol{z}_1)=\boldsymbol{z}_{i+1}$. Then from \eqref{eq4.2} we obtain $\boldsymbol{z}_1$ $\varepsilon$-shadows $\{\boldsymbol{x}_i\}_{i=1}^k$. 
Therefore, $(\rm{Orb}_{r}(X),\sigma_{rF})$ has shadowing property.

$(1)\Longrightarrow (3)$ is similar to the proof of $(1)\Longrightarrow (2)$, and $(3)\Longrightarrow (1)$ is similar to the proof of $(2)\Longrightarrow (1)$, we omit them.

(b) Suppose that $F_{inv}$ has shadowing property. Given $\varepsilon>0$, let $\delta>0$ witness $\varepsilon$-shadowing in the sense of Definition \ref{def3.2}. Since $F$ is open, $F^{-1}$ is continuous, and notice that $\rm{Orb}_{inv}(F)=\rm{Orb}_r(F^{-1})$. Hence, by Lemma \ref{lem3.12}, let $\delta_1>0$ witness the condition of Lemma \ref{lem3.12} with respect to $\delta$, that is, if $d(x,y)<\delta_1$ for any $x,y\in X$, then
\[\rho_H(\varphi^{-1}_0(x),\varphi^{-1}_0(y))<\delta,\]
where $\varphi_0$ is the first coordinate projection from $\rm{Orb}_{inv}(F)$ to $X$.

Let $\{x_0,\cdots,x_m\}$ be a $\delta_1$-pseudo-orbit. Consider $x_0$, there is $x_0'\in F(x_0)$ such that $d(x_0',x_0)<\delta_1$. Select $\overline{\boldsymbol{x}}_0'$ with $\varphi_0(\overline{\boldsymbol{x}}_0')=x_0'$ and $\varphi_1(\overline{\boldsymbol{x}}_0')=x_0$. Then there exists a point $\overline{\boldsymbol{x}}_1\in \varphi^{-1}_0(x_1)$ such that $\rho(\overline{\boldsymbol{x}}_0',\overline{\boldsymbol{x}}_1)<\delta$. Put $\overline{\boldsymbol{x}}_0=\sigma_{rF^{-1}}(\overline{\boldsymbol{x}}_0')$. Then it is easy to verify that 
\[\rho(F_{inv}(\overline{\boldsymbol{x}}_0),\overline{\boldsymbol{x}}_1)<\rho(\overline{\boldsymbol{x}}_0',\overline{\boldsymbol{x}}_1)<\delta.\]
Consider $x_1$ and $x_2$, there is $x_1'\in F(x_1)$ such that $d(x_1',x_2)<\delta_1$. By adding $x_1'$ to the first coordinate of $\overline{\boldsymbol{x}}_1$, we obtain a new point $\overline{\boldsymbol{x}}_1'$ in $\rm{Orb_{inv}}(F)$, then there is $\overline{\boldsymbol{x}}_2\in \rm{Orb}_{inv}(F)$ such that $\rho(\overline{\boldsymbol{x}}_1',\overline{\boldsymbol{x}}_2)<\delta$, which means that 
\[\rho(F_{inv}(\overline{\boldsymbol{x}}_1),\overline{\boldsymbol{x}}_2)<\rho(\overline{\boldsymbol{x}}_1,\overline{\boldsymbol{x}}_2)<\delta.\]
By induction, we can construct $\overline{\boldsymbol{x}}_i\in \varphi^{-1}_0(\overline{\boldsymbol{x}}_i)$ such that $\rho(F_{inv}(\overline{\boldsymbol{x}}_{i-1}),\overline{\boldsymbol{x}}_i)<\delta$ for all $i=1,\cdots,m$. Thus, there is a finite $F_{inv}$-orbit $\{\overline{\boldsymbol{y}}_0,\cdots,\overline{\boldsymbol{y}}_m\}$ such that
\[\rho(\overline{\boldsymbol{x}}_i,\overline{\boldsymbol{y}}_i)<\varepsilon\text{ for all }i=0,\cdots,m,\].
Moreover, $\varphi_0(\overline{\boldsymbol{y}}_0),\cdots,\varphi_0(\overline{\boldsymbol{y}}_m)$ is an $F$-orbit such that 
\[d(x_i,\varphi_0(\overline{\boldsymbol{y}}_i))<2\rho(\overline{\boldsymbol{x}}_i,\overline{\boldsymbol{y}}_i)<2\varepsilon.\]
As $\varepsilon$ can be chosen arbitrarily, $F$ has shadowing property.

Conversely, suppose $F$ has shadowing property. Let $\varepsilon>0$ and $\varepsilon_1>0$ witness the condition of Lemma \ref{lem3.12} with respect to $\varepsilon/3$, that is, if $d(x,y)<\varepsilon_1$, then 
\begin{align}\label{eq4.3}
    \rho_H(\varphi_0^{-1}(x),\varphi_0^{-1}(y))<\frac{\varepsilon}{3}.
\end{align}

Let $\delta>0$ witness the $\varepsilon_1/3$-shadowing. Given a $\delta/2$-pseudo-orbit $\{\overline{\boldsymbol{x}}_1,\cdots,\overline{\boldsymbol{x}}_m\}$ in $\rm{Orb}_{inv}(X)$, then
\begin{align}\label{eq4.4}
    \rho(F_{inv}(\overline{\boldsymbol{x}}_{i-1}),\overline{\boldsymbol{x}}_i)<\frac{\delta}{2}\text{ for all }i=1,\cdots,m.
\end{align}
Thus, 
\[d(F(\varphi_0(\overline{\boldsymbol{x}}_{i-1})),\varphi_0(\overline{\boldsymbol{x}}_{i}))<2\rho(F_{inv}(\overline{\boldsymbol{x}}_{i-1}),\overline{\boldsymbol{x}}_{i})<\delta\]
for all $i=1,\cdots,m$,
which means that $\{\varphi_0(\overline{\boldsymbol{x}}_{0}),\cdots,\overline{\boldsymbol{x}}_{m}\}$ is a $\delta$-pseudo-orbit of $F$. Then there is a finite $F$-orbit $(y_0,\cdots,y_m)$ such that 
\[d(\varphi_0(\overline{\boldsymbol{x}}_{i}),y_i)<\frac{\varepsilon_1}{3}.\]
Consider $\overline{\boldsymbol{x}}_{0}$, form \eqref{eq4.3}, there exists a point $\overline{\boldsymbol{y}}_{0}\in \varphi_0^{-1}(y_0)$ such that 
\[\beta_0=\rho(\overline{\boldsymbol{x}}_{0},\overline{\boldsymbol{y}}_{0})<\frac{\varepsilon}{3}.\]
For each $i=1,\cdots,m$, we add the word $y_i\cdots y_1$ to the front of $\overline{\boldsymbol{y}}_0$ to get new points $\overline{\boldsymbol{y}}_i\in \rm{Orb}_{inv}(X)$ and $\{\overline{\boldsymbol{y}}_0,\cdots,\overline{\boldsymbol{y}}_m\}$ is an $F_{inv}$-orbit. Then we can caculate that 
\begin{align*}
    \beta_1&=\rho(\overline{\boldsymbol{x}}_1,\overline{\boldsymbol{y}}_1)\\&<\frac{d(\varphi_0(\overline{\boldsymbol{x}}_1),y_1)}{2}+\frac{\rho(\sigma_{rF^{-1}}(\overline{\boldsymbol{x}}_1),\overline{\boldsymbol{x}}_0)+\rho(\overline{\boldsymbol{x}}_0,\overline{\boldsymbol{y}}_0)}{2}\\&<\frac{\varepsilon_1}{6}+\frac{\delta+\beta_0}{2}.
\end{align*}
By calculation the formula $\alpha_i=\varepsilon_1/6+(\delta+\alpha_{i-1})$, we can obtain
\begin{align*}
    \beta_i&=\rho(\overline{\boldsymbol{x}}_i,\overline{\boldsymbol{y}}_i)\\&<\frac{d(\varphi_0(\overline{\boldsymbol{x}}_i),y_i)}{2}+\frac{\rho(\sigma_{rF^{-1}}(\overline{\boldsymbol{x}}_i),\overline{\boldsymbol{x}}_{i-1})+\rho(\overline{\boldsymbol{x}}_{i-1},\overline{\boldsymbol{y}}_{i-1})}{2}\\&<\frac{\varepsilon_i}{6}+\frac{\delta+\beta_{i-1}}{2}\\&<\frac{1}{2^{i-1}}(\frac{\varepsilon}{3}-\frac{2\varepsilon_1}{3})+\frac{2\varepsilon}{3}\\&<\varepsilon,
\end{align*}
where $\beta_{i-1}=\rho(\overline{\boldsymbol{x}}_{i-1},\overline{\boldsymbol{y}}_{i-1})$. Therefore, $\{\overline{\boldsymbol{y}}_0,\cdots,\overline{\boldsymbol{y}}_m\}$ $\varepsilon$-shadows $\{\overline{\boldsymbol{x}}_0,\cdots,\overline{\boldsymbol{x}}_m\}$, which means that $F_{inv}$ has shadowing property.
\end{proof}

\begin{proposition}\label{prop3.19}
    Let $(X, F)$ be a continuous onto set-valued system. The following statements are equivalent:
    \begin{enumerate}[(a)]
        \item $F$ has the shadowing property.
        \item For any $\varepsilon > 0$, there exist $\delta > 0$ and $N \in \mathbb{N}$ such that for every $\delta$-pseudo-orbit $\{x_n\}$ satisfying 
        \[
        d(x_i^N, x_{i+N}) < \varepsilon/2, \quad \text{for some } x_i^N \in F^N(x_i),
        \]
        the sequence $\{x_{n+N}\}$ can be $\varepsilon$-shadowed by an $F$-orbit.
        \item For any $\varepsilon > 0$, there exist $\delta > 0$ and $N \in \mathbb{N}$ such that for every $\delta$-pseudo-orbit $\{x_n\}$ with 
        \[
        d(x_i^N, x_{i+N}) < \varepsilon/2, \quad \text{for some } x_i^N \in F^N(x_i),
        \]
        the sequence $\{x_n^N\}$ can be $\varepsilon$-shadowed by an $F$-orbit.
    \end{enumerate}
\end{proposition}

\begin{proof}
    $(a) \Longrightarrow (b)$. This follows directly from the definition of the shadowing property.

    $(b) \Longrightarrow (a)$. Let $\varepsilon > 0$ and $N>0$. Since $F:X\to 2^X$ is continuous, it is uniformly continuous, then we can find a sequence $\delta_1<\delta_2<\cdots<\delta_{N-1}<\delta_N<\varepsilon/2$ such that 
    \begin{align}\label{eqq4.5}
        d(x,y)<\delta_1+\delta_{i}\text{ implies }d_H(F(x),F(y))<\delta_{i+1}
    \end{align}
    and 
    \[\delta_N+\delta_1<\frac{\varepsilon}{2}.\]
    Now, let $\{x_0,\cdots,x_N,\cdots\}$ be a $\delta_1$-pseudo-orbit, then for each $i$ there is $x_i^{(1)}\in F(x_i)$ such that \[d(x_i^{(1)},x_{i+1})<\delta_1\] and $x_{i+1}'\in F(x_{i+1})$ such that \[d(x_{i+1}',x_{i+2})<\delta_1.\] From \eqref{eqq4.5}, as $d_H(F(x_i^{(1)}),F(x_{i+1}))<\delta_2$ we can find $x_i^{(2)}\in F(x_i^{(1)})$ with $d(x_i^{(2)},x_{i+1}')<\delta_2$, so \[d(x_i^{(2)},x_{i+2})<d(x_i^{(2)},x_{i+1}')+d(x_{i+1}',x_{i+2})<\delta_1+\delta_2.\] By induction, suppose $d(x_{i}^{(j-1)},x_{i+j-1})<\delta_1+\delta_{j-1}$ has obtained for some $j<N$, as $d(F(x_{i+j-1}),x_{i+j})<\delta_1$, there is a point $x_{i+j-1}'\in F(x_{i+j-1})$. From \eqref{eqq4.5}, as  $d_H(F(x_{i}^{(j-1)}),F(x_{i+j-1}))<\delta_{j}$, we can find $x_i^{(j)}\in F(x_i^{(j-1)})$ such that $d(x_i^{(j)},x_{i+j-1}')<\delta_j$, so 
    \[d(x_i^{(j)},x_{i+j})<d(x_i^{(j)},x_{i-1+j}')+d(x_{i-1+j}',x_{i+j})<\delta_{j}+\delta_1,\] for each $j=1,\cdots,N$.
    Thus, for each $i\in \mathbb{N}$ there is $x_i^N\in F^N(x_i)$ such that 
    \[d(x_i^N,x_{i+N})<\frac{\varepsilon}{2}.\]
    Thus, $\delta_1$ is the number witnessing the condition of (b). Also, since $F$ is onto, we extend the sequence $\{x_n\}$ by adding a finite orbit points $(x_{-N},\cdots,x_{-1})$ to the front of $\{x_0,\cdots,x_N,\cdots\}$, we obtain $\{x_{-N},\cdots,x_{-1},x_0,\cdots,x_N,\cdots\}$. Therefore, $\{x_0,\cdots,x_N,\cdots\}$ can be $\varepsilon$-shadowed by an $F$-orbit.

    $(b) \Longrightarrow (c)$. Let $\varepsilon>0$ and $\delta>0$ and $N\in\mathbb{N}$ witness the condition of (b) with respect to $\varepsilon/4$. Given a $\delta$-pseudo-orbit $\{x_n\}$, then $\{x_{n+N}\}$ can be $\varepsilon/4$-shadowed by some $F$-orbit $(z_n)$ and 
    \[d(x_i^N,z_n)<d(x_i^N,x_{i+N})+d(x_{i+N},z_i)<\frac{3\varepsilon}{4}.\]
    Thus, $(c)$ holds.

    $(c) \Longrightarrow (b)$. It follows a similar discussion as in $(b) \Longrightarrow (c)$.
\end{proof}

\begin{theorem}
    Let $(X, F)$ be a continuous onto set-valued system. Then $F$ has shadowing property if and only if $F_{inv}$ has shadowing property.
\end{theorem}

\begin{proof}
Suppose that $F_{inv}$ has shadowing property. Let $\varepsilon>0$ and $\delta>0$ witness the $\varepsilon/2$-shadowing for $F_{inv}$. Fix $N\in \mathbb{N}$ such that 
  \begin{align} \label{3.20.7}
        \sum_{n=N}^\infty \frac{1}{2^{n+1}} < \delta/2.
    \end{align}

    To complete the proof, we use a constructive approach as follows:

    Consider $x_0$ and a point $x_0^1\in F(x_0)$. Since $F$ is continuous, there is a sequence $0<\delta_1<\cdots<\delta_N$ satisfying:
    \begin{enumerate}[(1)]
        \item If $d(x_0^1,x_1)<\delta_1$ then $d_H(F(x_0^1),F(x_1))<\delta_2$ for all $x_0^1,x_1\in X$. Consequently, if $x_1^1\in F(x_1)$ with $d(x_1^1,x_2)<\delta_1$, then there is a point $x_0^2\in F(x_0^1)$ (hence, $x_0^2\in F^2(x_0)$) such that 
        \[d(x_0^2,x_1^1)<\delta_2\text{ and }d(x_0^2,x_2)<d(x_0^2,x_1^1)+d(x_1^1,x_2)<\delta_1+\delta_2.\]
        \item If $d(x_0^2,x_1^1)<\delta_2$, then $d_H(F(x_0^2),F(x_1^1))<\delta_3$ for any $x_0^2,x_1^1\in X$. Consequently, if $x_2^1\in F(x_2)$ with $d(x_2^1,x_3)<\delta_1$, as $d(x_1^1,x_2)<\delta_1$ in $(1)$, there exists $x_1^2\in F(x_1^1)$ (hence, $x_1^2\in F^2(x_1)$) such that 
        \[d(x_1^2,x_2^1)<\delta_2.\]
        Further, from $(1)$, since $d(x_0^2,x_1^1)<\delta_2$, there is $x_0^3\in F(x_0^2)$ (hence, $x_0^3\in F^3(x_0)$) such that 
        \[d(x_0^3,x_1^2)<\delta_3\text{ and }d(x_0^3,x_3)<\delta_1+\delta_2+\delta_3,\]
        \[d(x_1^2,x_3)<\delta_1+\delta_2.\]
        \item Consider the start point $x_0$, we end the procedure until we obtain $(x_0^Nx_0^{N-1}\cdots x_0^1x_0)$ such that $d(x_0^i,x_i)<\sum_{j=1}^i\delta_j$. Actually, we have the following graph:

          
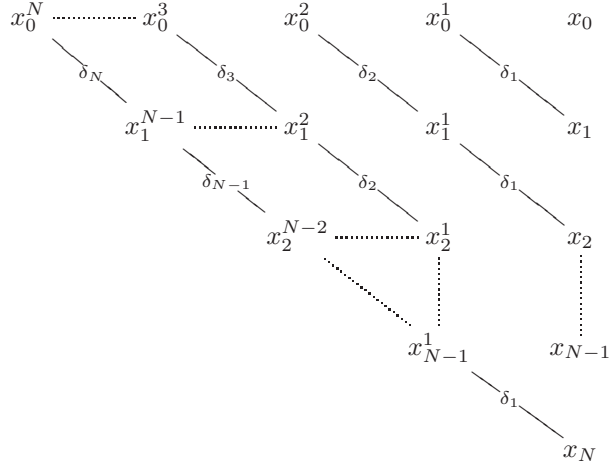
\begin{equation*}
        \xymatrix{
        x_0^N \ar@{-}[dr]|{\delta_N} \ar@{.}[r] & x_0^3 \ar@{-}[dr]|{\delta_3} & x_0^2 \ar@{-}[dr]|{\delta_2} & x_0^1 \ar@{-}[dr]|{\delta_1} & x_0 \\
        & x_1^{N-1} \ar@{-}[dr]|{\delta_{N-1}} \ar@{.}[r] & x_1^2 \ar@{-}[dr]|{\delta_2} & x_1^1 \ar@{-}[dr]|{\delta_1} & x_1 \\
        && x_2^{N-2} \ar@{.}[r] \ar@{.}[dr] & x_2^1 \ar@{.}[d] & x_2 \ar@{.}[d] \\
        &&& x_{N-1}^1 \ar@{-}[dr]|{\delta_1} & x_{N-1} \\
        &&&& x_N
        }
    \end{equation*}
    Then we continue the procedure to consider start point $x_i$ ($i\geq 2$) we can construct a sequence of points $(x_i^Nx_i^{N-1}\cdots x_i^1x_i)$ such that 
    \begin{enumerate}[(i)]
        \item $d(x_i^j,x_{i+1}^{j-1})<\delta_j$ for all $i\geq 0$ and $N-1\geq j\geq 1$.
        \item $d(x_i^j,x_{i+j})<\sum_{t=1}^j\delta_t$ for all $N-1\geq j\geq 1$.
    \end{enumerate}
    \end{enumerate}

    Since $F$ is continuous, we can choose $\delta_j$ small enough such that $\sum_{j=1}^{N-1}\delta_j<\delta/2$.

    Now, given a $\delta_1$-pseudo-orbit $\{x_i\}$ of $F$. By repeating the process above we can construct $\overline{\boldsymbol{x}}_i$ with $\varphi_j(\overline{\boldsymbol{x}}_i)=x_i^{N-j-1}$ for $j=0,\cdots,N-1$ and $\varphi_N(\overline{\boldsymbol{x}}_i)=x_i$ ensuring 
    \begin{enumerate}[(1')]
        \item $\rho(F_{inv}(\overline{\boldsymbol{x}}_i),\overline{\boldsymbol{x}}_{i+!})<\delta$ (due to (i) and \eqref{3.20.7}).
        \item $d(x_i^{N-1},x_{i+N-1})<\frac{\delta}{2}<\frac{\varepsilon}{8}$ (due to (3)).
    \end{enumerate}
    Since $F_{inv}$ has shadowing property, there is an $F_{inv}$-orbit $(\overline{\boldsymbol{y}}_i)$ such that 
    \[\rho(\overline{\boldsymbol{x}}_i,\overline{\boldsymbol{y}}_i)<\frac{\varepsilon}{2}.\]
    Notice that $\varphi_0(\overline{\boldsymbol{y}}_i)$ is an $F$-orbit. This implies that 
    \[d(\varphi_0(\overline{\boldsymbol{x}}_i),\overline{\boldsymbol{y}}_i)=d(y_i,x_{i}^{N-1})<\varepsilon.\]
    Therefore, by Proposition \ref{prop3.19}, $F$ has shadowing property.

   Conversely, suppose that $F$ has shadowing property. Let $\varepsilon>0$ and $\varepsilon_1=\varepsilon/8$ and $0<\delta_1<\varepsilon_1$ witness $\varepsilon_!$-shadowing. Given a $\delta_1/4$-pseudo-orbit $\{\overline{\boldsymbol{x}_i}\}$ of $F_{inv}$. Then $\{\varphi_0(\overline{\boldsymbol{x}}_i)\}$ is a $\delta_1/2$-pseudo-orbit of $F$. Thus, there is an $F$-orbit $\boldsymbol{y}\in \rm{Orb}_r(X)$ such that 
   \[d(\varphi_0(\overline{\boldsymbol{x}}_i),\pi_i(\boldsymbol{y}))<\varepsilon_1\text{ for all }i\in \mathbb{N}.\]

   Consider $y_0=\pi_0(\boldsymbol{y})$, as $F$ is onto, we can get a point $\overline{\boldsymbol{y}}_0\in \rm{Orb}_{inv}(X)$ with $\varphi_0(\overline{\boldsymbol{y}}_0)=y_0$, then 
   \[\beta_0=\rho(\overline{\boldsymbol{x}}_0,\overline{\boldsymbol{y}}_0)<\frac{d(\varphi_0(\overline{\boldsymbol{x}}_0),y_0)}{2}+\frac{1}{2}<\frac{\varepsilon_1}{2}+\frac{1}{2}.\]
   Since $\rho(\overline{\boldsymbol{x}}_1,F_{inv}(\overline{\boldsymbol{x}}_0))<\frac{\delta_1}{4}$, we have $\rho(\sigma_{rF^{-1}}(\overline{\boldsymbol{x}}_1),\overline{\boldsymbol{x}}_0)<\frac{\delta_1}{2}$, recall that 
   $\sigma_{rF^{-1}}:\rm{Orb}_{inv}(X)\to \rm{Orb}_{inv}(X)$, $(x_0x_1\cdots)\mapsto (x_1x_2\cdots)$. Thus, by adding $y_1$ to the first coordinate of $\overline{\boldsymbol{y}}_0$, we obtain a new point $\overline{\boldsymbol{y}}_1\in \rm{Orb}_{inv}(X)$, and get  
   \begin{align*}
       \beta_1=\rho(\overline{\boldsymbol{x}}_1,\overline{\boldsymbol{y}}_1)<\frac{d(y_1,\varphi_0(\overline{\boldsymbol{x}}_1))}{2}+\frac{\rho(\sigma_{rF^{-1}}(\overline{\boldsymbol{x}}_1),\overline{\boldsymbol{x}}_0)+\rho(\overline{\boldsymbol{x}}_0,\overline{\boldsymbol{y}}_0)}{2}
   \end{align*}
Now, suppose that we have obtained $\overline{\boldsymbol{y}}_{i-1}\in \rm{Orb}_{inv}(X)$, and an estimation of $\beta_{i-1}=\rho(\overline{\boldsymbol{x}}_{i-1},\overline{\boldsymbol{y}}_{i-1})$, then we add $y_i$ to the first coordinate to $\overline{\boldsymbol{y}}_{i-1}$ and obtain a new point $\overline{\boldsymbol{y}}_i\in \rm{Orb}_{inv}(X)$ and get
   \begin{align*}
       \beta_i=\rho(\overline{\boldsymbol{x}}_i,\overline{\boldsymbol{y}}_i)&<\frac{d(y_i,\varphi_0(\overline{\boldsymbol{x}}_i))}{2}+\frac{\rho(\sigma_{rF^{-1}}(\overline{\boldsymbol{x}}_i),\overline{\boldsymbol{x}}_{i-1})+\rho(\overline{\boldsymbol{x}}_{i-1},\overline{\boldsymbol{y}}_{i-1})}{2}\\&=\frac{\varepsilon_1}{2}+\frac{\delta_1/2+\beta_{i-1}}{2}.
   \end{align*}
   Hence, by calculating the formula
   \[\alpha_i=\frac{\varepsilon_1}{2}+\frac{\delta_1}{4}+\frac{\alpha_{i-1}}{2},\] we obtain
\[\beta_i\leq \frac{1}{2^i}(\frac{1}{2}-\frac{\varepsilon_1}{2}-\frac{\delta_1}{4})+\varepsilon_1+\frac{\delta_1}{2}.\]
Then there is $N$ big enough such that 
$\beta_i=\rho(\overline{\boldsymbol{x}}_i,\overline{\boldsymbol{y}}_i)<\varepsilon$ for all $i>N$.
Since $F$ is continuous, $F_{inv}$ is also continuous. Therefore, by choosing $\delta_1$ small enough, the number $N$ and every $\delta_1$-pseudo-orbit meet the condition of Proposition \ref{prop3.19}. Hence, $F_{inv}$ has shadowing property.
   
\end{proof}

\bibliographystyle{amsplain}

\medskip
Received xxxx 20xx; revised xxxx 20xx; early access xxxx 20xx.
\medskip

\end{document}